\documentclass[smallextended]{elsarticle}

\usepackage{latexsym,enumerate,tikz,float}
\usepackage{amsmath,amsthm,amsopn,amstext,amscd,amsfonts,amssymb}

\usetikzlibrary{decorations}
\usetikzlibrary{decorations.pathreplacing}

\newtheorem{theorem}{Theorem}[section]
\newtheorem{lemma}[theorem]{Lemma}
\newtheorem{proposition}[theorem]{Proposition}
\newtheorem{corollary}[theorem]{Corollary}

\newtheorem{problem}{Problem}

\newtheorem{openprob}{Open Problem}

\journal{}

\begin{document}

\begin{frontmatter}

  \title{Total $k$-domination in Cartesian product of complete graphs}

  \address[a]{Department of Mathematics and Statistics, Florida International University, 11200 SW 8th Street, Miami, FL 33199 U.S.A.}

  \author[a]{Walter Carballosa \corref{x}}
  \ead{wcarball@fiu.edu} \cortext[x]{Corresponding author.}

  \author[a]{Justin Wisby}
  \ead{jwisby@fiu.edu}

  \begin{abstract}
    Let $G=(V,E)$ be a finite undirected graph. A set $S$ of vertices in $V$ is said to be total $k$-dominating if every vertex in $V$ is adjacent to at least $k$ vertices in $S$.
    The total $k$-domination number, $\gamma_{kt}(G)$, is the minimum cardinality of a total $k$-dominating set in $G$.
    In this work we study the total $k$-domination number of Cartesian product of two complete graphs which is a lower bound of the total $k$-domination number of Cartesian product of two graphs. We obtain new lower and upper bounds for the total $k$-domination number of Cartesian product of two complete graphs. Some asymptotic behaviors are obtained as a consequence of the bounds we found. In particular, we obtain that $\displaystyle\liminf_{n\to\infty}\frac{\gamma_{kt}(G\Box H)}{n}\leq 2\,\left(\left\lceil\frac{k}{2}\right\rceil^{-1}+\left\lfloor\frac{k+4}{2}\right\rfloor^{-1}\right)^{-1}$ for graphs $G,H$ with order at least $n$. We also prove that the equality is attained if and only if $k$ is even. The equality holds when $G,H$ are both isomorphic to the complete graph, $K_n$, with $n$ vertices.
    Furthermore, we obtain closed formulas for the total $2$-domination number of Cartesian product of two complete graphs of whatever order. Besides, we prove that, for $k=3$, the inequality above is improvable to $\displaystyle\liminf_{n\to\infty} \gamma_{3t}(K_n\Box K_n)/n \leq 11/5$.
  \end{abstract}

\end{frontmatter}

\section{Introduction}
\label{sec:intro}
We begin by stating the terminology.
Throughout this paper, $G=(V,E)$ denotes a simple graph of order $|V|=n$. 
We denote two adjacent vertices $u$ and $v$ by $u\sim v$. For a nonempty set $X\subseteq V$ and a vertex $v\in V$ the degree of $v$ in $ X$ will be denoted by $d_{X}(v)=\left|\{u\in X: u\sim v\}\right|.$

A set $S$ of vertices in $V$ is said to be $k$-\emph{dominating} if  every vertex $v\in V \setminus S$ satisfies $d_S(v) \ge k$. The $k$-\emph{domination number}, $\gamma_k(G)$, is the minimum cardinality of a $k$-dominating set in $G$.
A set $S\subset V$ is said to be \emph{total $k$-dominating} if every vertex in $V$ is adjacent to at least $k$ vertices in $S$.
The \emph{total $k$-domination number}, $\gamma_{kt}(G)$, is the minimum cardinality of a total $k$-dominating set in $G$.
The notion of total domination in a graph was introduced by Cockayne, Dawes and Hedetniemi in \cite{CDH}. The total domination number ($k=1$) has also been studied in Cartesian product of graphs in \cite{BHKM,HR}.

The most famous open problem about domination in graphs is the Vizing's conjecture, see \cite{Vi}. This conjecture states that the dominating number of the Cartesian product of two graphs is greater than or equal the product of the dominating number of both factor graphs. Domination and some well-known variations have continuously been studied, see \emph{e.g.} \cite{ BAEK,Ha,HHS,He,KCB} and the references therein.
We recall that the Cartesian product of two graphs $G = \big(V (G),E(G)\big)$ and
$H = \big(V (H),E(H)\big)$ is the graph $G \Box H = (V,E)$, such that $V = \{(u,v) : u \in V(G),v \in V(H)\}$ and two vertices $(u_1,v_1),(u_2,v_2) \in V$ are adjacent in $G\Box H$ if and only if, either the case where $u_1= u_2$ and $v_1\sim v_2$ or the case where $v_1=v_2$ and $u_1\sim u_2$.
From this definition, it follows that the Cartesian product of two graphs is commutative.
When we refer to the Cartesian product of complete graphs, $K_n\Box K_m$, we denote $V(K_n):=\{v_1\ldots,v_n\}$ and $V(K_m):=\{w_1\ldots,w_m\}$.

The first approach to domination in graph appears within the problem of the five queens, \emph{i.e.}, place five queens on a chessboard so that every free square is dominated by at least one queen\footnote{One sample solution is queens on d4, e7, f5, g8, h6.}. Note that the solutions to this problem are dominating sets in the graph whose vertices are the $64$ squares of the chessboard and vertices $a, b$ are adjacent if a queen may move from $a$ to $b$ in one move or queen occupies vertex $a$.
More recently, a problem on total domination appeared as Questions 3 of the 40\textsuperscript{th} International Mathematical Olympiad, which was equivalent to determining the total domination number of the Cartesian product of two path graphs with same even order, \emph{i.e.}, $\gamma_{t}(P_{2n}\Box P_{2n})$.
Several authors have studied the total domination of product of graphs like Cartesian, strong and lexicographic, see \emph{e.g.} \cite{BHS,BJS,CK,LH}.
Some of the works associated to domination number in graphs study a problem with a natural presentation in a rectangular board involving chess pieces, for instance, one of the famous problem relative to domination in graph has recently been solved. Michael Simkin has solved the $n$-queens problem, see \cite{Si}; it is to determine $\mathcal{Q}(n)$, the number of ways to place $n$ mutually non-threatening queens on an $n\times n$ board.

Sometimes, throughout this work, we refer to the following equivalent problem in an $n\times m$ board which could be a convenient tool to visualize and obtain $\gamma_{kt}(K_n\Box K_m)$ for $k\geq2$.

\begin{problem}\label{prob:board}
Determine the minimum number of chess-rooks\footnote{Reminder that rooks, in the game of chess, are major pieces which may move horizontally or vertically to any other square in their rank (row) or file (column).} placed at distinct squares of an $n\times m$ board such that each cell is dominated by at least $k$ rooks considering that no rook dominated the square where it is placed and no rook obstructs other rooks.
\end{problem}

Clearly, the solution of Problem \ref{prob:board} is $\gamma_{kt}(K_n\Box K_m)$, and consequently, each rook configuration that gives a solution of the problem is a minimum total $k$-dominating set of $K_n\Box K_m$.
Furthermore, a total $k$-dominating set of $K_n\Box K_m$ provides a configuration of rooks that satisfies the Problem \ref{prob:board}, too.
Note that for every two graphs $G,H$ with orders $n,m$, respectively, $\gamma_{kt}(K_n\Box K_m)$ is a natural lower bound of $\gamma_{kt}(G\Box H)$ for every $k\ge1$ since $G\Box H \subseteq K_n\Box K_m$, \emph{i.e.}, $\gamma_{kt}(K_n\Box K_m) \le \gamma_{kt}(G\Box H)$. Indeed, every closed formula obtained in this work is a sharp lower bound for the Cartesian product of two graphs with respective orders.

In this paper we discuss the total $k$-domination number of Cartesian product of two complete graphs $\gamma_{kt}(K_n\Box K_m)$ for $k\ge2$.
In Section \ref{sect:k}, we study $\gamma_{kt}(K_n\Box K_m)$ for every $k\ge2$. In Section \ref{sect:k=2}, we deal with the particular case $k=2$, and in Section \ref{sect:k=3} we deal with $k=3$.
Throughout this paper, we obtain lower and upper bounds for $\gamma_{kt}(K_n\Box K_m)$ for $k\ge2$ that improve the bounds in \cite{BSS}, see, \emph{e.g.}, Theorems \ref{th:rows-cols}, \ref{th:k_even}, \ref{th:nm}, and \ref{prop:nn3}.
In addition, we obtain a double recurrence formula given in Theorems \ref{th:+4} and \ref{th:+13} with initial conditions given in Theorems \ref{th:nn} and \ref{th:nn+1} which easily calculates  $\gamma_{2t}(K_n\Box K_m)$ for every $n,m\ge2$, see Theorem \ref{th:nm} and Table \ref{closed_formula}.
Although, we didn't find closed formulas for $\gamma_{3t}(K_n\Box K_m)$ we obtain some sharp inequalities that might show that $k=3$ ($k$ odd, resp.) looks like to be distinct to the results for $k=2$ ($k$ even, resp.). The main result in Section \ref{sect:k=3} are associated to the asymptotic behavior of $\gamma_{3t}(K_n\Box K_n)/n$ as $n$ goes to infinity.
Furthermore, the results in this work deduce asymptotic behavior for $\gamma_{kt}(K_n\Box K_m)$ in comparison to $n$, see Theorems \ref{th:rows-colsnxn}, \ref{th:behavior_nm_even}, \ref{th:lowerbound}, \ref{th:behaviornm} and \ref{th:behaviork3}.

\
\section{Total $k$-domination number of $K_n\Box K_m$}\label{sect:k}

In order to obtain main results, we first collect some results through technical lemmas that will prove useful.

\begin{lemma}\label{l:UpperB}
  For every $2\le k\le n\le m$,
  \begin{equation}\label{eq:UpperB}
    \left\lfloor \frac{k(n+1)}2\right\rfloor\,+1 \le \gamma_{kt}(K_n\Box K_m) \le k\,n.
  \end{equation}
\end{lemma}

\begin{proof}
  Without loss of generality, we can assume that $m\ge n$ as the Cartesian product of graphs commute. On the one hand, consider $S\subset V(K_n\Box K_m)$ a total $k$-dominating set of $K_n\Box K_m$.
  Since $d_{K_n\Box K_m}(v)=n+m-2$ for every $v\in S$ and $d_{S}(u)\ge k$ for every $u\in V(K_n\Box K_m)$, we have $|S| (n+m-2) \ge k\,nm$.
  Thus, we have
  \[
    |S| \ge \frac{k\,nm}{n+m-2} =  \frac{k}2 \,(n+1) +    \frac{k(m-n)(n-1)+2k}{2(n+m-2)}\ge \frac{k(n+1)}2 +\frac{k}{n+m-2}.
  \]
  In order to obtain the second inequality, it suffices to choice $S:= V(K_n)\times \{w_1,\ldots,w_k\}$ for distinct vertices $w_1,\ldots,w_k$ in $K_m$, since $S$ is a total $k$-dominating set of $K_n\Box K_m$.
\end{proof}

\begin{lemma}\label{l:rooks}
  In every rook configuration of a Problem \ref{prob:board} solution for $2\le k\le n\le m$, there is a row or a column with at least $\left\lceil\frac{k+2}{2}\right\rceil$ rooks.

  Furthermore, if $\gamma_{kt}(K_n\Box K_m)<k\,\min\{n,m\}$ then there is at least one rook in each row and column.
\end{lemma}

\begin{proof}
  Consider a rook placed in a square on board. The square must be dominated by at least $k$ other rooks, thus the number of rooks in its row plus the number of rooks in its column must be at least $k+2$, so the result follows.

  On the other hand, if there is a row (column, resp.) with no rook then the squares in that row (column, resp.) must be dominated by at least $k$ distinct rooks.
\end{proof}

The following result states that in every configuration of rooks satisfying the condition of Problem \ref{prob:board} contains at least $\gamma_{kt}(K_r\Box K_s)$ rooks within every $r$ rows and $s$ columns for $r,s\ge k$.

\begin{lemma}\label{l:rows_files}
  Let $A,B$ be an $r$-set of $V(K_n)$ and an $s$-set of $V(K_m)$, respectively, with $2\le k< r\le n$ and $k< s\le m$. If $S$ is a total $k$-dominating set of $K_n\Box K_m$, then $\left|S \cap \big[ (A\times V(K_m))\cup(V(K_n)\times B) \big] \right| \ge \gamma_{kt}(K_r\Box K_s)$.
\end{lemma}

\begin{proof}
  Without loss of generality, we can assume that $A=\{v_1,\ldots,v_r\}$ and $B=\{w_1,\ldots,w_s\}$.
  Note that if $r=n$ and or $s=m$ the result is obvious.
  Thus we can assume that $r<n$ and $s<m$.
  Let us consider $V_{1}:= \{v_1,\ldots,v_r\}\times\{w_1,\ldots,w_s\}$, $V_{2}:= \{v_1,\ldots,v_r\}\times\{w_{s+1},\ldots,w_{m}\}$, $V_{3}:=\{v_{r+1},\ldots,v_{n}\} \times\{w_1,\ldots,w_s\}$ and $V:=V_1\cup V_2\cup V_3$.
  Without loss of generality, we can assume that $r\le s$.

  On the one hand, assume that $S\cap V_1=\emptyset$.
  Hence, since $d_{S\cap(V_2\cup V_3)}\big((v_i,w_i)\big)=d_{S}\big((v_i,w_i)\big)\ge k$ for every $1\le i\le r$ and $\left\{N\big((v_i,w_i)\big)\cap\big(V_2\cup V_3\big)\right\}_{i=1}^{r}$ is a set of pairwise disjoint subsets of $V_2\cup V_3$, we have $\left|S\cap(V_2\cup V_3)\right|\ge kr \ge\gamma_{kt}(K_r\Box K_s)$.

  On the other hand, assume that $S\cap V_1\neq\emptyset$.
  Define $f_i:=\left|S\cap \left[\{v_i\}\times\{w_1,\ldots,w_m\}\right]\right|$ for $1\le i\le r$ and $c_j:=\left|S\cap \left[\{v_1,\ldots,v_n\}\times\{w_j\}\right]\right|$ for $1\le j\le s$.
  Without loss of generality we can assume that $f_1\le f_2\le\ldots\le f_r$ and $c_1\le c_2\le\ldots\le c_s$.
  Note that $f_i+c_j\ge k$ for every $1\le i\le r$ and $1\le j\le s$; moreover, if $(v_i,w_j)\in S$ we have $f_i+c_j\ge k+2$.
  Besides, if $f_1+c_1>k$, then $f_i+c_j>k$ for every $1\le i\le r$ and every $1\le j\le s$, and consequently, we can swap every element $(v_x,w_y)$ in $S\cap V_2$ ($S\cap V_3$, resp.) with an element in $\overline{S}\cap V_1$ remaining the total $k$-domination in $V_1$.


  \begin{figure}[H]
    \centering
    \begin{tikzpicture}[x=.75in, y=.75in]
      \draw[thick] (0,0)rectangle(3,3);
      \draw[line width=1.5pt,dashed] (0,2)--(3,2) (1,3)--(1,0);
      \draw[thick] (0,1.25)--(3,1.25) (1.75,3)--(1.75,0);
      \draw[line width=.75pt,decorate,decoration={brace,amplitude=5pt}] (.025,3.03) -- node[above=5pt] {$q-1$} (.975,3.03);
      \draw[line width=.75pt,decorate,decoration={brace,amplitude=5pt}] (1.025,3.03) -- node[above=5pt] {$s-q+1$} (1.725,3.03);
      \draw[line width=.75pt,decorate,decoration={brace,amplitude=5pt}] (-.03,2.025) -- node[left=5pt] {$p-1$} (-.03,2.975);
      \draw[line width=.75pt,decorate,decoration={brace,amplitude=5pt}] (-.03,1.275) -- node[left=5pt] {$r-p+1$} (-.03,1.975);
    \end{tikzpicture}
    \caption{Auxiliary subdivision for a total ${k}$-dominating set in $K_n\Box K_m$.} \label{fig_Lemma}
  \end{figure}

  Now we assume that $f_1+c_1=k$. Since $S\cap V_1\neq\emptyset$ there exists $1< p\le r$ such that  $f_1=f_{p-1}<f_p$, or $1< q\le s$ such that $c_1=c_{q-1}<c_q$. Thus, we have a configuration like the one in Figure \ref{fig_Lemma}. Consider that both $p$ and $q$ exist, and define

  \begin{table}[H]
    \centering
    \begin{tabular}{l l}
      $V_{11}:=\{v_1,\ldots,v_{p-1}\}\times\{w_1,\ldots,w_{q-1}\}$ , & $V_{12}:=\{v_1,\ldots,v_{p-1}\}\times\{w_q,\ldots,w_{s}\}$, \\
      $V_{21}:=\{v_p,\ldots,v_{r}\}\times\{w_1,\ldots,w_{q-1}\}$,    &
      and $V_{22}:=\{v_p,\ldots,v_{r}\}\times\{w_q,\ldots,w_{s}\}$\\
    \end{tabular}
  \end{table}

  Perhaps, some of the set $V_{12},V_{21},V_{22}$ can be empty-set, but not all since at least $p$ or $q$ exists. Note that $S\cap V_{11}=\emptyset$ since $f_i+c_j= k$ for every $1\le i<p$ and every $1\le j<q$. Therefore, we can swap elements $(v_x,w_y)$ in $S\cap V_2$ ($S\cap V_3$, resp.) with elements in $\overline{S}\cap (V_{12}\cup V_{22})$ ($\overline{S}\cap (V_{21}\cup V_{22})$, resp.) remaining the total $k$-domination in $V_1$. Notice that if the swapping described above empties both $V_2$ and $V_3$, then the result follows.
  Define $S^\prime$ as the new total $k$-total dominating set (at least on $V_1$) obtained after the swapping described above.
  Assume the swapping left a remainder in $V_2$ ($V_3$, resp.), then $V_{22}\subseteq S$ and $V_{12}\subseteq S$ ($V_{21}\subseteq S$, resp.).
  Note that if $p=2$ or $q=2$, the result follows since $r,s>k$ and $V_{11}$ is included in a row or a column.
  Assume now $p,q>2$. Note that if $S^\prime\cap (\{v_{r+1},\ldots,v_{n}\}\times\{w_1,\ldots,w_{q-1}\})\neq\emptyset$, then $S^\prime\cap (\{v_{r+1},\ldots,v_{n}\}\times\{w_j\})$ has the same cardinality, $x$, for every $1\le j\le q-1$, so we can swap those elements in $S^{\prime}$ with the elements in $x$ rows in $V_{11}$ remaining the total $k$-domination in $V_1$. Define the new total $k$-dominating set (at least on $V_1$) by $S^{\prime\prime}$.
  Notice that the case when $S^\prime\cap (\{v_{r+1},\ldots,v_{n}\}\times\{w_1,\ldots,w_{q-1}\})\neq\emptyset$, is analogous. Therefore, $S^{\prime\prime}\cap V_1$ is a total $k$-dominating set in $V_1$.
\end{proof}

Lemma \ref{l:rows_files} has the following result as a direct consequence.

\begin{theorem}\label{th:monotomy}
  For every $k\ge2$, if $n_1\leq n_2$ and $m_1\le m_2$, then we have
  \begin{equation}\label{eq:monotony}
    \gamma_{kt}(K_{n_1}\Box K_{m_1})\leq \gamma_{kt}(K_{n_2}\Box K_{m_2})
  \end{equation}
\end{theorem}

\begin{proof}
  Using Lemma \ref{l:rows_files} we have that $\gamma_{kt}(K_{n_1}\Box K_{m_2})\leq \gamma_{kt}(K_{n_2}\Box K_{m_2})$, by letting $A=\{v_1,v_2,\ldots,v_{n_1}\}$ and $B=\{w_1,w_2,\ldots,w_{m_2}\}$ in $K_{n_2}\Box K_{m_2}$.
  Similarly, by Lemma \ref{l:rows_files} we have $\gamma_{kt}(K_{n_1}\Box K_{m_1})\leq \gamma_{kt}(K_{n_1}\Box K_{m_2})$.
  Therefore,
  \[
    \gamma_{kt}(K_{n_1}\Box K_{m_1})\leq \gamma_{kt}(K_{n_2}\Box K_{m_2})
  \]
\end{proof}

The follow result is a consequence of Lemmas \ref{l:UpperB} and \ref{l:rows_files}.

\begin{proposition}\label{prop:1xfloor}
  For every $2\le k\le n\le m$,
  \[
    \gamma_{kt}(K_{n+1}\Box K_{m+\lfloor\frac{k}2\rfloor+1})\ge \gamma_{kt}(K_n\Box K_m) + \left\lfloor\frac{k}2\right\rfloor+1.
  \]
\end{proposition}

\begin{proof}
  Let $S$ be a minimum total $k$-dominating set of $K_{n+1}\Box K_{m+\lfloor\frac{k}2\rfloor+1}$, so $|S|=\gamma_{kt}(K_{n+1}\Box K_{m+\lfloor\frac{k}2\rfloor+1})$. By Lemma \ref{l:UpperB}, there is a rows in a configuration given by $S$ with at least $\left\lfloor \frac{k}2\right\rfloor+1$ rooks. Without loss of generality we can assume that $(v_{n+1},w_{m+1}),(v_{n+1},w_{m+2}),\ldots,(v_{n+1},w_{m+\lfloor\frac{k}2\rfloor+1}) \in S$. Then, by applying Lemma \ref{l:rows_files} to the first $n$ rows and first $m$ columns, we obtain the result.
\end{proof}

Using a similar reasoning as above in Lemma \ref{l:rows_files} we obtain the following result.

\begin{theorem}\label{th:k_to_k+1}
  For every $n,m\ge k\ge2$
  \begin{equation}\label{eq:k_to_k+1}
    \gamma_{(k+1)t}(K_n\Box K_m)\le\gamma_{kt}(K_n\Box K_m)+n.
  \end{equation}
\end{theorem}

\begin{proof}
  Let $S$ be a minimum total $k$-dominating set of $K_n\Box K_m$.
  Define $f_i:=\left|S\cap \left[\{v_i\}\times\{w_1,\ldots,w_m\}\right]\right|$ for $1\le i\le n$ and $c_j:=\left|S\cap \left[\{v_1,\ldots,v_n\}\times\{w_j\}\right]\right|$ for $1\le j\le m$.
  Note that $f_i+c_j\ge k$ for every $1\le i\le n$ and $1\le j\le m$; moreover, if $(v_i,w_j)\in S$, $f_i+c_j\ge k+2$; and if $f_i+c_j= k+1$, then $(v_i,w_j)\notin S$.
  Without loss of generality, we can assume that there exist $p,q\ge1$ and $r,s,t,u\ge0$ such that $f_1=\ldots= f_p$, $f_p+1=f_{p+1}=\ldots= f_{p+r}$, $f_{p+r}+1=f_{p+r+1}\le\ldots\le f_{p+r+t}$, $f_{p+r+t}<f_{p+r+t+1}\le\ldots\le f_{n}$, and $c_1=\ldots= c_q$, $c_q+1=c_{q+1}=\ldots= c_{q+s}$, $c_{q+s}+1=c_{q+s+1}\le\ldots\le c_{q+s+u}$, $c_{q+s+u}<c_{q+s+u+1}\le\ldots\le c_{m}$. We define the following regions.

  
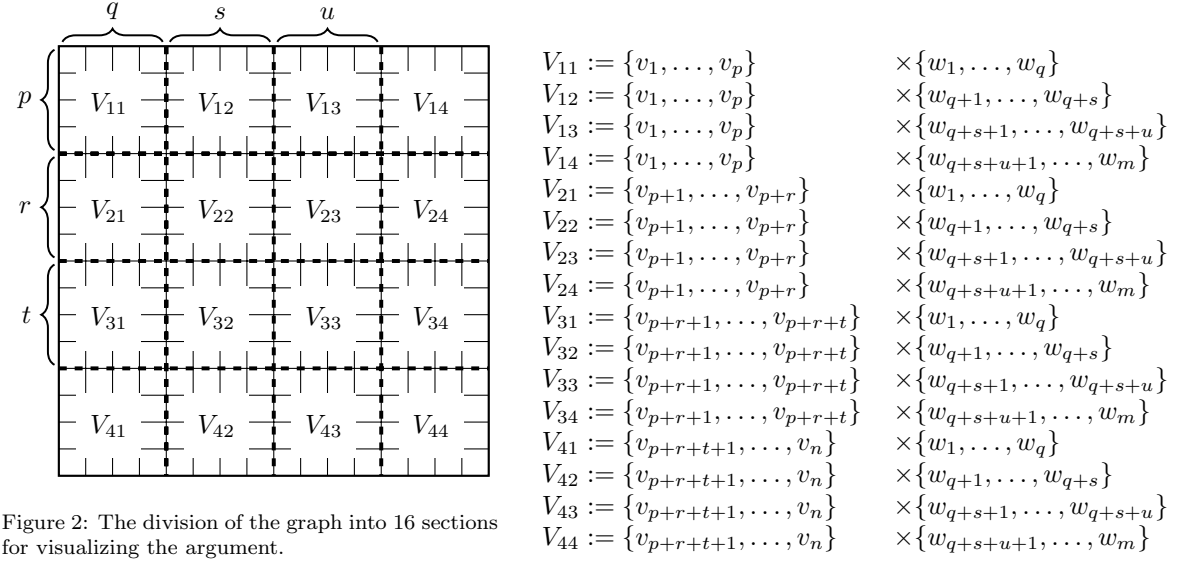
\begin{figure}[H]
    \centering
    \begin{tikzpicture}[x=.75in, y=.75in]
      \draw[thick] (0,0)rectangle(3,3);
      \draw[line width=1.5pt,dashed]
      (0,2.25)--(3,2.25) (.75,3)--(.75,0)
      (0,2.25)--(3,2.25) (.75,3)--(.75,0)
      (0,1.5)--(3,1.5) (1.5,3)--(1.5,0)
      (0,.75)--(3,.75) (2.25,3)--(2.25,0)
      ;
      \foreach \x in {
          1,2,3,4
        }{
          \foreach \y in {
              1,2,3,4
            }{
              \fill[white,shift={(-.375+.75*\x,3.375-.75*\y)}] (-.25,-.25)rectangle node{\textcolor{black}{$V_{\y \x}$}}(.2,.2);
            }
        }
      \foreach[count=\ki] \k in {q,s,u}{
          \draw[line width=.75pt,decorate,decoration={brace,amplitude=5pt}] (-.725+.75*\ki,3.03) -- node[above=5pt] {$\k$} (-.025+.75*\ki,3.03);
        }
      \foreach[count=\ki] \k in {p,r,t}{
          \draw[line width=.75pt,decorate,decoration={brace,amplitude=5pt}] (-.03,3.025-.75*\ki) -- node[left=5pt] {$\k$} (-.03,3.725-.75*\ki);
        }
    \end{tikzpicture}
    \caption{The division of the graph into 16 sections for visualizing the argument.}
    \label{fig:k_to_k+1}
  \end{figure}
  Some of the sets $V_{xy}$ with $1\le x,y\le 4$ can be an empty-set. Figure \ref{fig:k_to_k+1} shows an auxiliary view of such an arrangement.
  Besides, $|S\cap V_{11}|=|S\cap V_{12}|=|S\cap V_{21}|=0$ and all vertices in $V_{14}$, $V_{23}$, $V_{24}$, $V_{32}$, $V_{33}$, $V_{34}$, $V_{41}$, $V_{42}$, $V_{43}$, $V_{44}$ are total dominated by at least $k+1$ vertices in $S$.
  Thus, $S^\prime:=S\cup(\{v_1\}\times\{w_1,w_2,\ldots,w_{q+s+u}\})$ dominates with at least $k+1$ vertices to all vertices in $K_n\Box K_m$ except to $S\cap(\{v_1\}\times\{w_{q+s+1},\ldots,w_{q+s+u}\})$ which $S$ dominates with exactly $k$ vertices. Therefore, we can obtain $S^{\prime\prime}$ by adding to $S^\prime$ as many elements as in $S\cap(\{v_1\}\times\{w_{q+s+1},\ldots,w_{q+s+u}\})$ within its corresponding column. It is easy to check that $S^{\prime\prime}$ is a total $(k+1)$-dominating set of $K_n\Box K_m$.

  \begin{figure}[H]
    \centering
    \begin{tikzpicture}[x=1in, y=1in]
      \draw[thin,step=.25] (0,0)grid(1,1);
      \draw[thick] (0,0)rectangle(1,1);
      \draw[line width=2pt,dashed] (0,.25)--(1,.25) (.75,0)--(.75,1) (2,.25)--(3,.25) (2.75,0)--(2.75,1);
      \draw[thin,step=.25] (2,0)grid(3,1);
      \draw[thick] (2,0)rectangle(3,1);

      \foreach \x in {0,...,2}{
          \fill (.125+.25*\x,.125) circle(.075in);
          \fill (2.125+.25*\x,.125) circle(.075in);
          \fill (2.875,.375+.25*\x) circle(.075in);
          \fill (.875,.375+.25*\x) circle(.075in);
          \fill[color=gray] (2.125+.25*\x,.875) circle(.075in);
          \fill[color=gray] (2.875,.125) circle(.075in);
        }
    \end{tikzpicture}
    \caption{Minimum total $k$-dominating configurations for $K_4\Box K_4$ from $k=2$ to $k=3$.} \label{fig:3-to-4}
  \end{figure}

  The inequality in \eqref{eq:k_to_k+1} is sharp, for instance, $\gamma_{2t}(K_4 \Box K_4)=6$ and $\gamma_{3t}(K_4 \Box K_4)=10$. Figure \ref{fig:3-to-4} shows a sharp case of the inequality \eqref{eq:k_to_k+1}
\end{proof}

\begin{theorem}\label{th:rows-cols}
  For every $2\le k< n\le m$, we have
  \[
    \gamma_{kt}(K_n\Box K_m)\leq \left\lceil\frac{k}{2}\right\rceil\left(n+m-2\left\lceil\frac{k}{2}\right\rceil+2\right)-1.
  \]
\end{theorem}

\begin{proof}
  Start with the grid representation of $K_n\Box K_m$ and rooks populating the first $\left\lceil\frac{k}{2}\right\rceil$ rows and $\left\lceil\frac{k}{2}\right\rceil$ columns, except squares $(i,j)$ with $1\le i,j<\left\lceil\frac{k}{2}\right\rceil$ or $i,j>\left\lceil\frac{k}{2}\right\rceil$; like appear in Figure \ref{fig:rows-cols}.
  Therefore, we can easily verify that every square is $k$ dominated.

  \begin{figure}[H]
    \centering
    \begin{tikzpicture}[x=1in, y=1in]
      \draw[thin,step=.25] (0,0)grid(3,3);
      \draw[thick] (0,0)rectangle(3,3);
      \draw[line width=2pt,dashed] (0,2)--(1,2)--(1,3);
      \foreach \x/\y in {
          $c_1$/0,
          $\cdots$/.25,
          $c_{\left\lceil\frac{k}{2}\right\rceil}$/.75,
          $c_m$/2.75
        }{
          \node[anchor=south] at (\y +.125,3) {\x};
        }
      \foreach \x/\y in {
          $f_1$/0,
          $\vdots$/.25,
          $f_{\left\lceil\frac{k}{2}\right\rceil}$/.75,
          $f_n$/2.75
        }{
          \node[anchor=east] at (0,2.875-\y) {\x};
        }
      \foreach \x in {0,...,8}{
          \fill (.875+.25*\x,2.875) circle(.075in);
          \draw (.875+.25*\x,2.65) node {$\vdots$};
          \draw (.875+.25*\x,2.4) node {$\vdots$};
          \fill (.875+.25*\x,2.125) circle(.075in);
          \fill (.125,.125+.25*\x) circle(.075in);
          \fill (.375,.125+.25*\x) node {$\cdots$};
          \fill (.625,.125+.25*\x) node {$\cdots$};
          \fill (.875,.125+.25*\x) circle(.075in);
        }
    \end{tikzpicture}
    \caption{Total $k$-dominating configuration for $K_n\Box K_m$.} \label{fig:rows-cols}
  \end{figure}
\end{proof}

What we notice from the figure above is that, in most cases, the configuration can be improved. This prompts the following corollary by leaving free of rooks a square of size
$\left\lceil\frac{k}{2}\right\rceil\times\left\lceil\frac{k}{2}\right\rceil$ in the upper-left corner and no rooks on a rectangular block of size $\left(n-\left\lceil\frac{k}{2}\right\rceil\right)\times\left(m-\left\lceil\frac{k}{2}\right\rceil\right)$ in the lower-right corner when $n,m>k+1$ with rooks populating all other squares. This configuration of rooks gives the following results.

\begin{corollary}\label{cor:rows-cols}
  For every $k\ge2$ and $n, m\ge k+2$, we have
  \[
    \gamma_{kt}(K_n\Box K_m)\leq \left\lceil\frac{k}{2}\right\rceil\left(n+m-2\left\lceil\frac{k}{2}\right\rceil\right).
  \]
\end{corollary}

Notice that above configuration given by Corollary \ref{cor:rows-cols} is a total $k$-dominating sets and they are minimal by inclusion for $n=k+2$. Besides, we can prove that they are solution of Problem \ref{prob:board} for every $k\ge2$.

\begin{proposition}\label{prop:n=k+2}
  For every $k\ge2$
  \begin{equation}\label{eq:n=k+2}
    \gamma_{kt}(K_{k+2}\Box K_{k+2})=2\,\left\lceil\frac{k}2\right\rceil\,\left(k+2-\left\lceil\frac{k}2\right\rceil\right).
  \end{equation}
\end{proposition}

\begin{proof}
  The configuration described above, given by Corollary \ref{cor:rows-cols} for $n=m= k+2$, gives that
  \[
    \gamma_{kt}(K_n\Box K_m)\leq 2\,\left\lceil\frac{k}{2}\right\rceil\left(k+2-\left\lceil\frac{k}{2}\right\rceil\right).
  \]
  In order to prove the other, reversed, inequality, consider $S$ a minimum total $k$ dominating set of $K_{k+2}\Box K_{k+2}$ and consider $\{f_i\}_{i=1}^{k+2}$ and $\{c_j\}_{j=1}^{k+2}$ like in proof of Lemma \ref{l:rows_files}, \emph{i.e.}, $f_i:=\left|S\cap \left[\{v_i\}\times\{w_1,\ldots,w_{k+2}\}\right]\right|$ for $1\le i\le k+2$ and $c_j:=\left|S\cap \left[\{v_1,\ldots,v_{k+2}\}\times\{w_j\}\right]\right|$ for $1\le j\le k+2$.
  Without loss of generality we can assume that $f_1\le f_2\le\ldots\le f_{k+2}$ and $c_1\le c_2\le\ldots\le c_{k+2}$.
  Recall $f_i+c_j\ge k$ for every $1\le i,j\le k+2$; moreover, if $(v_i,w_j)\in S$ we have $f_i+c_j\ge k+2$.
  We separate the proof into two cases, when $k$ is even and odd, respectively.

  \bigskip

  Let's consider $k=2r$ for some $r\ge1$, \emph{i.e.}, $k$ is even.
  Seeking for a contradiction, we assume that
  \[
    |S|< 2\,\left\lceil\frac{k}{2}\right\rceil\left(k+2-\left\lceil\frac{k}{2}\right\rceil\right)-1 = r\,(k+2) + 2r-1.
  \]
  Note that if $f_1<r$ ($c_1<r$, resp.) then $c_j\ge r+1$ ($f_j\ge r+1$, resp.) and $|S|=\displaystyle\sum_{j=1}^{k+2} c_j\ge (r+1)\,(k+2)$ \big($|S|=\displaystyle\sum_{i=1}^{k+2} f_i\ge (r+1)\,(k+2)$, resp.\big) which is a contradiction with $|S|< r\,(k+2) + (2r-1)$. So, $f_1\ge r$ and $c_1\ge r$. Besides, using the argument above of summation of number of rooks by rows or columns, we have $f_1,c_1<r+1$.
  Thus, we can assume $f_1=c_1=r$.
  Let $p\ge1$ such that $f_p=r$ and $f_{p+1}>r$, and let $q\ge p$ such that $f_q\le r+1$ and $f_{q+1}>r+1$.
  Since $c_1=r$ we have that the number of rows with at least $r+2$ rooks must be at least $r$, \emph{i.e.}, $k+2-q\ge r$ or $q\le r+2$. Then, we have
  \[
    |S|=\displaystyle\sum_{i=1}^{k+2} f_i = \sum_{i=1}^{q} f_i + \sum_{i=q+1}^{k+2} f_i \ge q\,r + (k+2-q)\,(r+2) \ge r\,(k+2) + 2r>|S|.
  \]
  That is the contradiction we were looking for and then, we have the result when $k$ is even.

  \bigskip

  Let's consider $k=2r+1$ for some $r\ge1$, \emph{i.e.}, $k$ is odd.
  Seeking for a contradiction we assume that
  \[
    |S|< 2\,\left\lceil\frac{k}{2}\right\rceil\left(k+2-\left\lceil\frac{k}{2}\right\rceil\right)-1 = (r+1)\,(k+2) + r.
  \]
  Note that if $f_1<r$ ($c_1<r$, resp.) then $c_j\ge r+2$ ($f_j\ge r+2$, resp.) and $|S|=\displaystyle\sum_{j=1}^{k+2} c_j\ge (r+2)\,(k+2)$ \big($|S|=\displaystyle\sum_{i=1}^{k+2} f_i\ge (r+2)\,(k+2)$, resp.\big) which is a contradiction with $|S|< (r+1)\,(k+2) + r$. So, $f_1\ge r$ and $c_1\ge r$, but not both equal to $r$.
  Besides, using the argument above of summation of number of rooks by rows or columns, we have $f_1,c_1\le r+1$.
  Thus, without loss of generality we can assume $f_1=r$ and $c_1=r+1$.
  Let $p\ge1$ such that $c_p=r+1$ and $c_{p+1}> r+1$, and let $q\ge p$ such that $c_q\le r+2$ and $c_{q+1}>r+2$.
  Since $f_1=r$ we have that the number of rows with at least $r+3$ rooks must be at least $r$, \emph{i.e.}, $k+2-q\ge r$ or $q\le r+3$.
  Then, we have
  \[
    |S|=\displaystyle\sum_{i=1}^{k+2} c_i = \sum_{i=1}^{q} c_i + \sum_{i=q+1}^{k+2} c_i \ge q\,(r+1) + (k+2-q)\,(r+3) \ge (r+1)\,(k+2) + 2r>|S|.
  \]
  That is the contradiction we were looking for, and consequently, we have the result when $k$ is odd.
\end{proof}

In Section \ref{sect:k=2}, we obtain a closed formula for $\gamma_{2t}(K_n\Box K_m)$ when $k=2$, where the basic minimum configuration given by Corollary \ref{cor:rows-cols} for $n=m=4$ plays an important roll, together with the configuration given by Figure \ref{fig:rows-cols} for $n=m=3$. Section \ref{sect:k=3} deals with $k=3$; however, the configuration that plays the most important roll doesn't follow from the configuration given by Corollary \ref{cor:rows-cols}.

\begin{theorem}\label{th:rows-colsnxn}
  For every $k\ge2$, 
  \[
    \displaystyle\liminf_{n\to\infty}\frac{\gamma_{kt}(K_n\Box K_n)}{n}\leq 2\,\left(\frac{1}{\left\lceil\frac{k}{2}\right\rceil}+\frac{1}{\left\lfloor\frac{k+4}{2}\right\rfloor}\right)^{-1}.
  \]
\end{theorem}

We remark that the right side of the inequality in Theorem \ref{th:rows-colsnxn} is the harmonic mean of $\left\lceil\frac{k}2\right\rceil$ and $\left\lfloor\frac{k+4}2\right\rfloor$.

\begin{proof}
  Start with the grid representation of $K_{p(k+2)}\Box K_{p(k+2)}$ for some $p\ge1$. By Proposition \ref{prop:n=k+2} we have
  \[
    \gamma_{kt}(K_{k+2}\Box K_{k+2})=2\,\left\lceil\frac{k}2\right\rceil\,\left(k+2-\left\lceil\frac{k}2\right\rceil\right).
  \]
  Particularly, we may use a configuration like in Corollary \ref{cor:rows-cols} for $n=m=k+2$.
  Then, accommodate $p$ of those configurations with size $(k+2)\times(k+2)$ into the main diagonal of the grid.
  This configuration by blocks is a total $k$-dominating set of $K_{p(k+2)}\Box K_{p(k+2)}$, and consequently, we obtain that for every $p\ge1$
  \[
    \gamma_{kt}(K_{p(k+2)}\Box K_{p(k+2)})\leq 2p\,\left\lceil\frac{k}{2}\right\rceil\left(k+2-\left\lceil\frac{k}{2}\right\rceil\right)=2p\,\left\lceil\frac{k}2\right\rceil\,\left\lfloor\frac{k+4}2\right\rfloor.
  \]
  Therefore, we have
  \begin{equation}\label{eq:1}
    \displaystyle\frac{\gamma_{kt}(K_{p(k+2)}\Box K_{p(k+2)})}{p(k+2)} = \frac{2p\,\left\lceil\frac{k}{2}\right\rceil\left(k+2-\left\lceil\frac{k}{2}\right\rceil\right)}{p(k+2)}=\frac{2\,\left\lceil\frac{k}2\right\rceil\left\lfloor\frac{k+4}2\right\rfloor}{k+2}
  \end{equation}
  The result follows from \eqref{eq:1} by taking limit into the subsequence $n\in \{p(k+2)\}_{p=1}^{\infty}$ and the identity $x+2=\left\lceil\frac{x}2\right\rceil+\left\lfloor\frac{x+4}2\right\rfloor$ for every $x\in\mathbb{N}$.
\end{proof}

\begin{figure}[H]
  \centering
  \begin{tikzpicture}[x=.5in, y=.5in]
    \draw[thick] (0,0)rectangle(5,4.5);
    \draw[line width=1pt,dashed]
    (0,1.25)--(1.25,1.25)--(1.25,0)
    (1.75,1.75)--(3,1.75)--(3,3)--(1.75,3)--(1.75,1.75)
    (3,0)--(3,4.5) (0,3)--(5,3);
    \fill (.5,0) rectangle (1.25,0.5)
    (2.25,1.75) rectangle (2.25+.75,1.75+.5)
    (3,3) rectangle (3+.75,3+.5)
    (4.25,4) rectangle (4.25+.75,4+.5)
    (0,.5) rectangle (.5,.5+.75)
    (1.75,2.25) rectangle (1.75+.5,2.25+.75);
    \foreach \x in {3,4,5}{\fill (1.25+\x/16,1.25+\x/16) circle (1pt);}
    \foreach \x in {3,4,5}{{\fill (3.75+\x/16,3.5+\x/16) circle (1pt);}}
  \end{tikzpicture}
  \caption{Total $k$-dominating configuration for a $\left[p(k+2)+r\lceil\frac{k}2\rceil\right]\times \left[p(k+2)+r\lfloor\frac{k+4}2\rfloor\right]$ board.} \label{fig:2p+r}
\end{figure}

Using a configuration of rooks populating into a $\left[p(k+2)+r\lceil\frac{k}2\rceil\right]\times \left[p(k+2)+r\lfloor\frac{k+4}2\rfloor\right]$ board for every $p,r\ge0$ but no both zeros, like in Figure \ref{fig:2p+r}, \emph{i.e.}, $2p+r$ blocks of size $\left\lceil\frac{k}2\right\rceil\,\left\lfloor\frac{k+4}2\right\rfloor$ full of rooks placed $p+r$ horizontally and $p$ vertically, we can obtain the following upper bound.

\begin{proposition}\label{prop:nxm_pr}
  For every $k\ge2$
  \begin{equation}\label{eq:nm_pr}
    \gamma_{kt}(K_{p(k+2)+r\lceil\frac{k}2\rceil}\Box K_{p(k+2)+r\lfloor\frac{k+4}2\rfloor})\leq (2p+r)\,\left\lceil\frac{k}2\right\rceil\,\left\lfloor\frac{k+4}2\right\rfloor.
  \end{equation}
\end{proposition}

The following results correspond to a general minimum rooks configuration for a $p(k+2)\times p(k+2)$ board, $p\ge1$, and it allows to obtain sharp asymptotic behavior for $\gamma_{kt}(K_n\Box K_n)$, too. See Proposition \ref{prop:nxm_p} and Theorem \ref{th:behavior_nm_even}.

\begin{proposition}\label{prop:nxm_p}
  For every even number $k\ge2$ and every natural number $p\ge1$
  \begin{equation}\label{eq:n=p(k+2)}
    \gamma_{kt}(K_{p(k+2)}\Box K_{p(k+2)})\, = p\,k\,\left(\frac{k}2+2\right).
  \end{equation}
\end{proposition}

\begin{proof}
  Proposition \ref{prop:nxm_pr} when $r=0$ gives that
  \[
    \gamma_{kt}(K_{p(k+2)}\Box K_{p(k+2)})\leq 2p\,\left\lceil\frac{k}2\right\rceil\,\left\lfloor\frac{k+4}2\right\rfloor.
  \]

  \bigskip

  Let's now consider $S$ a minimum total $k$ dominating set of $K_{p(k+2)}\Box K_{p(k+2)}$ and consider $\{f_i\}_{i=1}^{p(k+2)}$ and $\{c_j\}_{j=1}^{p(k+2)}$ like in proof of Lemma \ref{l:rows_files}, \emph{i.e.}, $f_i:=\left|S\cap \left[\{v_i\}\times\{w_1,\ldots,w_{p(k+2)}\}\right]\right|$ for $1\le i\le p(k+2)$ and $c_j:=\left|S\cap \left[\{v_1,\ldots,v_{p(k+2)}\}\times\{w_j\}\right]\right|$ for $1\le j\le p(k+2)$.
  Without loss of generality we can assume that $f_1\le f_2\le\ldots\le f_{p(k+2)}$ and $c_1\le c_2\le\ldots\le c_{p(k+2)}$.
  Recall $f_i+c_j\ge k$ for every $1\le i,j\le p(k+2)$; moreover, if $(v_i,w_j)\in S$ we have $f_i+c_j\ge k+2$.
  Let's consider $k=2r$ for some $r\ge1$, \emph{i.e.}, $k$ is even.

  \medskip

  Note that if $f_1>r$ or $c_1>r$, then $|S|$ is not minimum since in that case we have
  \[
    |S|=\displaystyle\sum_{i=1}^{p(k+2)} f_i=\sum_{j=1}^{p(k+2)} c_j\ge (r+1)\,p(k+2)=2p\cdot(r+1)^2 > 2p\cdot r(r+2)=2p\,\left\lceil\frac{k}2\right\rceil\,\left\lfloor\frac{k+4}2\right\rfloor.
  \]
  If $f_1<r$ ($c_1<r$, resp.) then $c_j\ge r+1$ ($f_j\ge r+1$, resp.) and $|S|$ is not minimum. Thus, we can assume that $f_1=c_1=r$.
  Let $a\ge1$ such that $f_a=r$ and $f_{a+1}>r$, and let $b\ge a$ such that $f_b\le r+1$ and $f_{b+1}>r+1$.
  Analogously, let $d\ge1$ such that $c_d=r$ and $c_{d+1}>r$, and let $e\ge d$ such that $c_e\le r+1$ and $c_{e+1}>r+1$.
  Let $V_{1,1}=\{(v_i,w_j\,|\,1\le i\le a,1\le j\le d)\}$, $V_{1,2}=\{(v_i,w_j\,|\,1\le i\le a,d< j\le e)\}$, $V_{1,3}=\{(v_i,w_j\,|\,1\le i\le a,e< j\le p(k+2))\}$, $V_{2,1}=\{(v_i,w_j\,|\,a< i\le b,1\le j\le d)\}$, $V_{2,2}=\{(v_i,w_j\,|\,a< i\le b,d< j\le e)\}$, $V_{2,3}=\{(v_i,w_j\,|\,a< i\le b,e< j\le p(k+2))\}$, $V_{3,1}=\{(v_i,w_j\,|\,b< i\le p(k+2),1\le j\le d)\}$, $V_{3,2}=\{(v_i,w_j\,|\,1\le i\le a,d< j\le e)\}$, and $V_{3,3}=\{(v_i,w_j\,|\,1\le i\le a,e< j\le p(k+2))\}$.
  Note that in some case could occur that $a=b$ or $d=e$. See Figure \ref{fig:p(k+2)} for an auxiliary view.

  \begin{figure}[H]
    \centering
    \begin{tikzpicture}[x=.75in, y=.75in]
      \draw[thick] (0,0)rectangle(3,3);
      \draw[line width=1.5pt,dashed] (0,2)--(3,2) (1,3)--(1,0);
      \draw[thick] (0,1.25)--(3,1.25) (1.75,3)--(1.75,0);
      \draw[line width=.75pt,decorate,decoration={brace,amplitude=5pt}] (.025,3.03) -- node[above=5pt] {$d$} (.975,3.03);
      \draw[line width=.75pt,decorate,decoration={brace,amplitude=5pt}] (1.025,3.03) -- node[above=5pt] {$e-d$} (1.725,3.03);
      \draw[line width=.75pt,decorate,decoration={brace,amplitude=5pt}] (-.03,2.025) -- node[left=5pt] {$a$} (-.03,2.975);
      \draw[line width=.75pt,decorate,decoration={brace,amplitude=5pt}] (-.03,1.275) -- node[left=5pt] {$b-a$} (-.03,1.975);
    \end{tikzpicture}
    \caption{Auxiliary subdivision for a total ${k}$-dominating set in $K_n\Box K_n$ when $n=p(k+2)$.} \label{fig:p(k+2)}
  \end{figure}

  Note that $S\cap V_{1,1}=S\cap V_{1,2}=S\cap V_{2,1}=\emptyset$ since $f_i+c_j\le k+1$ for corresponding $i,j$. Now we claim that $S\cap V_{2,2}=\emptyset$. Assume by contradiction that there is $(v,w)\in S\cap V_{2,2}$. Hence, $S\setminus\{(v,w)\}$ is also a total $k$-dominating set that contradicts $S$ is minimum. Thus, $S\cap V_{2,2}=\emptyset$. Using the same argument of minimal by inclusion of $S$, we have that $S\cap V_{2,3}=\emptyset$; so, by symmetry $S\cap V_{3,2}=\emptyset$, too. Therefore, by Lemma \ref{l:rooks} we have $a=b$ and $d=e$; moreover, we have that $V_{1,2}$, $V_{2,1}$, $V_{2,2}$, $V_{2,3}$ and $V_{3,2}$ actually vanished.  Furthermore, since $S$ is minimal by inclusion, we have $S\cap V_{3,3}=\emptyset$ as well.

  \medskip

  Then, we have
  \[
    |S|=\displaystyle\sum_{i=1}^{a} f_i + \sum_{j=1}^{d} c_j=(a+d) \,r=\sum_{i=a+1}^{p(k+2)} f_i + \sum_{j=d+1}^{p(k+2)} c_j \ge \big[2p(k+2)-(a+d)\big]\,(r+2).
  \]
  Then, we have $a+d\ge 2p(r+2)$, and consequently, $|S|\ge 2p\,r(r+2)$.
\end{proof}

Proposition \ref{prop:nxm_p} gives the following direct result which improves Corollary \ref{cor:rows-cols}.

\begin{theorem}\label{th:k_even}
  For every even number $k\ge2$ and every $m\ge n\ge k+2$
  \begin{equation}\label{eq:k_even}
    \displaystyle \gamma_{kt}(K_{n}\Box K_{m})\leq \frac{k}2\left(m+n-k\left\lfloor\frac{n}{k+2}\right\rfloor\right).
  \end{equation}
\end{theorem}

\begin{proof}
  Let $n=p(k+2)+r$ where $p\ge1$ and $0\le r<k+2$ are the quotient and remainder of the division $n$ divide by $k+2$.
  Then, starting with a minimum total $k$-dominating configuration for $p(k+2)\times p(k+2)$ like in Figure \ref{fig:2p+r} when $r=0$, \emph{i.e.}, configuration used in Proposition \ref{prop:nxm_p}, we may add $r$ rows ($r+m-n$ columns, resp.) enlarging one of the $\left\lceil\frac{k}2\right\rceil\times\left\lfloor\frac{k+4}2\right\rfloor$ \big($\left\lfloor\frac{k+4}2\right\rfloor\times\left\lceil\frac{k}2\right\rceil$, resp.\big) rectangles populating rooks until a $\left\lceil\frac{k}2\right\rceil\times\left(\left\lfloor\frac{k+4}2\right\rfloor+r\right)$ \Big($\left(\lfloor\frac{k+4}2+r+m-n\right)\rfloor\times\left\lceil\frac{k}2\right\rceil$, resp.\Big) rectangles we obtain a total $k$-dominating configuration for $n\times m$. Therefore, we have
  \[
    \begin{aligned}
      \displaystyle \gamma_{kt}(K_{n}\Box K_{m})
        & \leq 2p\,\frac{k}2\,\frac{(k+4)}2 + (2r+m-n)\,\frac{k}2 \\
        & \leq \frac{k}2\Big(p(k+4) + (2r+m-n)\Big)               \\
        & \leq \frac{k}2\left(m + 2p+r\right)                     \\
    \end{aligned}
  \]
\end{proof}

The following is a direct consequence of Theorem \ref{th:k_even}.
\begin{corollary}\label{cor:kn_even}
  For every even number $k\ge2$ and $m\ge n\ge k+2$, we have
  \[
    \gamma_{kt}(K_n\Box K_m)< kn, \quad \text{if  } m<n+k\left\lfloor\frac{n}{k+2}\right\rfloor.
  \]
\end{corollary}

The following result gives an asymptotic behavior of $\gamma_{kt}(K_n\Box K_n)$ for $k$ even when $n$ goes to infinity.

\begin{theorem}\label{th:behavior_nm_even}
  For every even number $k\ge2$ and a natural number $n\ge k$
  \begin{equation}\label{eq:behavior_nm_even}
    \displaystyle\lim_{n\to\infty}\, \frac{\gamma_{kt}(K_n\Box K_n)}{n} = \left(\frac{1}{k} + \frac{1}{k+4} \right)^{-1} 
  \end{equation}
\end{theorem}

\begin{proof}
  The result follows directly from the facts that $n=q_n(k+2) +r_n$ where $q_n$ and $0\le r_n<k+2$ are the quotient and the remainder when $n$ is divided by $k+2$. Thus, by Proposition \ref{prop:nxm_p} and Theorem \ref{th:k_even} we have
  \[
    \displaystyle \frac{\gamma_{kt}(K_n\Box K_n)}{n}=\frac{k\,q_n\,\frac{k+4}{2}+s_{n}}{q_n (k+2)+r_n}\quad \text{for some } 0\le s_{n}\le \frac{k(k+4)}{2}.
  \]
  Therefore, we have
  \[
    \displaystyle \frac{\gamma_{kt}(K_n\Box K_n)}{n} = \frac{q_n\,\frac{k(k+4)}{2}+s_{n}}{q_n (k+2)+r_n} = \frac{\frac{k(k+4)}{2}+\frac{s_{n}}{q_n}}{(k+2)+\frac{r_n}{q_n}}
    \xrightarrow{\quad n\to\infty\quad}
    \frac{k(k+4)}{2(k+2)}.
  \]
\end{proof}

\begin{corollary}
  Let $G,H$ be two graphs with order at least $n$. Then, for every even number $k\ge2$
  \[
    \displaystyle\liminf_{n\to\infty}\frac{\gamma_{kt}(G\Box H)}{n} = 2\,\left(\left\lceil\frac{k}{2}\right\rceil^{-1}+\left\lfloor\frac{k+4}{2}\right\rfloor^{-1}\right)^{-1}.
  \]
\end{corollary}

%
%

\
\section{Total $2$-domination number of $K_n\Box K_m$}\label{sect:k=2}

In this section we deals with the total $2$-domination number of $K_n\Box K_m$ for $n,m\ge2$. The main result in this section is Theorem \ref{th:nm} that gives a closed formula of $\gamma_{2t}(K_n\Box K_m)$ for every $n,m\ge2$, see Table \ref{closed_formula}.
We have the following consequence from Lemma \ref{l:UpperB} when $k=2$.

\begin{proposition}\label{prop:2xn}
  For every $m\ge2$ we have $\gamma_{2t}(K_2\Box K_m)=4$.
\end{proposition}

The following result is a particular version of Lemma \ref{l:rooks} for $k=2$ with an improvement.

\begin{lemma}\label{l:non-null}
  For every $n,m\ge2$, if $\gamma_{2t}(K_n\Box K_m)<2\,\min\{n,m\}$ then the following statements hold in every rook configuration of a Problem \ref{prob:board} solution
  \begin{enumerate}[(i)]
    \item there is a rook in each row (column, resp.),
    \item there is a row (column, resp.) with at least three rooks.
  \end{enumerate}
\end{lemma}

\begin{proof}
  {\it (of Part ii)} 
  Since $\gamma_{2t}(K_n\Box K_m)<2\,\min\{n,m\}$ there is a row (column, resp.) with just one rook, the square with that rook must be dominated by others two rooks located in the same column (row, resp.). 
\end{proof}

In fact, we have the following result as a consequence of Lemma \ref{l:non-null}. This result also appears in \cite[Proposition 3.2]{BSS} when $k=2$.

\begin{proposition}\label{prop:+m_2n}
  For every $n\le m$, we have $\gamma_{2t}(K_n\Box K_m)\ge\min\{m+2,2n\}$. Furthermore, $\gamma_{2t}(K_n\Box K_{m})=2n$ for $m\ge2n-2$.
\end{proposition}

Note that Lemma \ref{l:rows_files} for $k=2$ could be extended until $r,s\ge2$ (instead of $r,s>2$). It is easily seen that $|S\cap V|\ge 4$ when $r=2$.
In order to obtain the exact value of $\gamma_{2t}(K_n\Box K_n)$ we need
the following interesting result will be useful to obtain some of the main results of this work.

\begin{theorem}\label{th:-3o4}
  For every $6\le n\le m$, 
  \begin{equation}\label{eq:-3o4}
    \gamma_{2t}(K_n\Box K_m)\ge \min\{ \gamma_{2t}(K_3\Box K_3) + \gamma_{2t}(K_{n-3}\Box K_{m-3}) , \gamma_{2t}(K_4\Box K_4) + \gamma_{2t}(K_{n-4}\Box K_{m-4}) \}.
  \end{equation}
\end{theorem}

\begin{proof}
  Note that if $\gamma_{2t}(K_n\Box K_m)=2n$, then the inequality holds. Hence, we can assume that $\gamma_{2t}(K_n\Box K_m)<2n$. 
  Let $S$ be a minimum total $2$-dominating set of $K_n\Box K_m$.
  By Lemma \ref{l:non-null} there is a vertex $(v_i,w_j)\in V(K_n\Box K_m)$ such that $\left|S\cap \big(\{v_i\}\times V(K_m)\big)\right|\ge 3$ and $\left|S\cap \big( V(K_n)\times \{w_j\}\big)\right|\ge 3$.
  Without loss of generality, we can assume that $i=j=1$.
  Assume first that $(v_1,w_1)\notin S$. Without loss of generality, we can assume that $(v_1,w_2)$, $(v_1,w_3)$, $(v_1,w_4)$, $(v_2,w_1)$, $(v_3,w_1)$, $(v_4,w_1)\in S$, \emph{i.e.}, $S$ has the configuration in Figure \ref{fig:-3o4}, right.
  Denote by $A:=S\cap V_{1}$ where $V_{1}:= \{v_1,v_2,v_3,v_4\}\times\{w_1,w_2,w_3,w_4\}$.
  Clearly, $|A|\ge \gamma_{2t}(K_4\Box K_4)=6$.
  Hence, by Lemma \ref{l:rows_files} we have $|S\setminus A|\ge\gamma_{2t}(K_{n-4}\Box K_{m-4})$, and consequently, $\gamma_{2t}(K_n\Box K_m)\ge \gamma_{2t}(K_4\Box K_4) + \gamma_{2t}(K_{n-4}\Box K_{m-4})$.
  The proof when $(v_1,w_1)\in S$ is analogous.
  Note that $S$ has the configuration in Figure \ref{fig:-3o4}, left, and Lemma \ref{l:rows_files} gives $\gamma_{2t}(K_n\Box K_m)\ge \gamma_{2t}(K_3\Box K_3)+\gamma_{2t}(K_{n-3}\Box K_{m-3})$.
\end{proof}

\begin{figure}[H]
  \centering
  \begin{tikzpicture}[x=1in,y=1in]
    \def\SS{.25}
    \begin{scope}[xshift=0in]
      \draw[thin](0,0) grid[step=\SS in](8*\SS,8*\SS);
      \fill[white](8*\SS,8*\SS)rectangle(.9,.9);
      \foreach \x in {0,...,2}{
          \fill (\x*\SS+.125,.125)circle (.075in);
          \fill (.125,\x*\SS+.125)circle (.075in);
        }
      \draw[thick](0,3*\SS)--(8*\SS,3*\SS);
      \draw[thick](3*\SS,0)--(3*\SS,8*\SS);
      \draw[thick](0,0) rectangle(8*\SS,8*\SS);
    \end{scope}
    \begin{scope}[xshift=2.25in]
      \draw[thin](0,0) grid[step=\SS in](9*\SS,9*\SS);
      \fill[white](9*\SS,9*\SS)rectangle(1.15,1.15);
      \foreach \x in {1,...,3}{
          \fill (\x*\SS+.125,.125)circle (.075in);
          \fill (.125,\x*\SS+.125)circle (.075in);
        }
      \draw[thick](0,4*\SS)--(9*\SS,4*\SS);
      \draw[thick](4*\SS,0)--(4*\SS,9*\SS);
      \draw[thick](0,0) rectangle(9*\SS,9*\SS);
    \end{scope}
  \end{tikzpicture}

  \caption{Auxiliar configurations for Theorem \ref{th:-3o4}} \label{fig:-3o4}
\end{figure}

\begin{lemma}\label{l:+1+1}
  For every $n,m\ge2$, we have
  \begin{equation}\label{eq:+1+1}
    \gamma_{2t}(K_n\Box K_m)+1\le \gamma_{2t}(K_{n+1}\Box K_{m+1})\le \gamma_{2t}(K_n\Box K_m)+2.
  \end{equation}
\end{lemma}

\begin{proof}
  Let $S'$ be a total $2$-dominating set of $K_{n+1}\Box K_{m+1}$ and consider $(v',w')\in S'$. By Lemma \ref{l:rows_files} we have $\gamma_{2t}(K_{n+1}\Box K_{m+1})-1=|S'\setminus\{(v',w')\}|\ge \gamma_{2t}(K_n\Box K_m)$, and so, the first inequality in \eqref{eq:+1+1} holds.
  Let $S$ be a total $2$-dominating set of $K_n\Box K_m$. If $|S|=\gamma_{2t}\big(K_n\Box K_m\big)=2\min\{n,m\}$ then $\gamma_{2t}\big(K_{n+1}\Box K_{m+1}\big)\le2\min\{n+1,m+1\}=|S|+2$. Assume that $|S|<2\min\{n,m\}$.
  Then by Lemma \ref{l:non-null} there are vertices $v\in V(K_n)$ and $w\in V(K_m)$ such that $\left|S\cap\big(\{v\}\times V(K_m)\big)\right|\ge3$ and $\left|S\cap\big(V(K_n)\times\{w\}\big)\right|\ge3$. Thus it is a simple matter to check that $S\cup\{(v,w_{m+1}),(v_{n+1},w)\}$ is a total $2$-dominating set of $K_{n+1}\Box K_{m+1}$.
\end{proof}

Lemma \ref{l:+1+1} and Theorem \ref{th:-3o4} have the following direct consequence.

\begin{theorem}\label{th:+4}
  For every $n, m\ge2$, if there is a minimum total $2$-dominating set $S$ of $K_{n}\Box K_{m}$ such that $S\cap\big(\{v\}\times V(K_m)\big)\neq\emptyset$ and $S\cap\big(V(K_n)\times\{w\}\big)\neq\emptyset$ for every $v\in V(K_n)$ and $w\in V(K_m)$ then
  \begin{equation}\label{eq:+4}
    \gamma_{2t}(K_{n+4}\Box K_{m+4}) = \gamma_{2t}(K_{n}\Box K_{m})+6.
  \end{equation}
\end{theorem}

Figure \ref{fig:67-68} left shows a minimal configuration for a total $2$-dominating set of $K_{6}\Box K_{7}$ which satisfies \eqref{eq:+4}; however, it does not verify the condition of Theorem \ref{th:+4} since any total $2$-dominating set $S$ of $K_{2}\Box K_{3}$ with $S\cap\big(\{v\}\times V(K_3)\big)\neq\emptyset$ and $S\cap\big(V(K_2)\times\{w\}\big)\neq\emptyset$ for every $v\in V(K_2)$ and $w\in V(K_3)$ is non-minimal.
Similarly, Figure \ref{fig:67-68} right shows a non-minimal configuration for a total $2$-dominating set of $K_{6}\Box K_{8}$ which does not verify neither the condition of Theorem \ref{th:+4} nor the equality in \eqref{eq:-3o4} since $\gamma_{2t}(K_6\Box K_8)=11\neq10=\min\{6+4,5+6\}$.

\begin{figure}[H]
  \centering
  \begin{tikzpicture}[x=1in,y=1in]
    \def\SS{.25}
    \begin{scope}[xshift=0in]
      \draw[thin](0,0) grid[step=\SS in](7*\SS,6*\SS);
      \foreach \x in {1,...,4}{\fill (.125,\x*\SS+.125)circle (.075in);}
      \foreach \x in {1,...,3}{\fill (\x*\SS+.125,.125)circle (.075in);}
      \foreach \x in {4,...,6}{\fill (\x*\SS+.125,1.375)circle (.075in);}
      \draw[very thick](0,4*\SS)--(7*\SS,4*\SS);
      \draw[very thick](4*\SS,0)--(4*\SS,6*\SS);
      \draw[very thick](0,0) rectangle(7*\SS,6*\SS);
    \end{scope}
    \begin{scope}[xshift=2in]
      \draw[thin](0,0) grid[step=\SS in](8*\SS,6*\SS);
      \foreach \x in {1,...,3}{\fill (.125,\x*\SS+.125)circle (.075in);}
      \foreach \x in {1,...,3}{\fill (\x*\SS+.125,.125)circle (.075in);}
      \foreach \x in {4,...,7}{\fill (\x*\SS+.125,1.375)circle (.075in);}
      \foreach \x in {4,...,5}{\fill (\x*\SS+.125,1.125)circle (.075in);}
      \draw[very thick](0,4*\SS)--(8*\SS,4*\SS);
      \draw[very thick](4*\SS,0)--(4*\SS,6*\SS);
      \draw[very thick](0,0) rectangle(8*\SS,6*\SS);
    \end{scope}
  \end{tikzpicture}
  \caption{Minimal configuration for $6\times7$ (left) and non-minimal for $6\times8$ (right).} \label{fig:67-68}
\end{figure}

\begin{theorem}\label{th:+13}
  For every $3\leq n\leq m$, if $\gamma_{2t}(K_{n}\Box K_{m})<2n$ then
  \begin{equation}\label{eq:+13}
    \gamma_{2t}(K_{n+1}\Box K_{m+3}) = \gamma_{2t}(K_{n}\Box K_{m})+3.
  \end{equation}
\end{theorem}

\begin{proof}
  Let $S$ be a minimum total $2$-dominating set of $K_{n+1}\Box K_{m+3}$. By Lemma \ref{l:non-null} there is $v\in V(K_{n+1})$ such that $\left|S\cap\{v\}\times V(K_{m+3})\right| \geq3$. Without loss of generality, we can assume that $A=\{(v_{n+1},w_{m+1})$, $(v_{n+1},w_{m+2})$, $(v_{n+1},w_{m+3})\}\subseteq S$. Then, by Lemma \ref{l:rows_files}, we have $\gamma_{2t}(K_{n+1}\Box K_{m+3})-3=|S\setminus A|\ge \gamma_{2t}(K_n\Box K_m)$, and so, the inequality $\gamma_{2t}(K_{n+1}\Box K_{m+3})\ge \gamma_{2t}(K_n\Box K_m)+3$ holds.
  Let $S^\prime$ be a total $2$-dominating set of $K_n\Box K_m$. Thus it is a simple matter to check that $S^\prime\cup A$ is a total $2$-dominating set of $K_{n+1}\Box K_{m+3}$ obtaining $\gamma_{2t}(K_{n+1}\Box K_{m+3})\leq \gamma_{2t}(K_n\Box K_m)+3$.
\end{proof}


Now, we approach the case $n=m$, \emph{i.e.}, to compute $\gamma_{2t}(K_n\Box K_n)$. The proof of the following Proposition is recommended to the reader.

\begin{proposition}\label{prop:n=3-6}
  We have $\gamma_{2t}(K_2\Box K_2)=4$, $\gamma_{2t}(K_3\Box K_3)=5$, $\gamma_{2t}(K_4\Box K_4)=6$, $\gamma_{2t}(K_5\Box K_5)=8$ and $\gamma_{2t}(K_6\times K_6)=10$.
\end{proposition}


\begin{figure}[H]
  \centering
  \begin{tikzpicture}[x=1in,y=1in]
    \def\SS{.25}
    \begin{scope}[xshift=0in]
      \draw[thin](0,0) grid[step=\SS in](2*\SS,2*\SS);
      \foreach \x in {0,...,1}{
          \fill (\x*\SS+.125,.125)circle (.075in);
          \fill (.125,\x*\SS+.125)circle (.075in);
        }
      \fill (.375,.375)circle (.075in);
      \draw[very thick](0,0) rectangle(2*\SS,2*\SS);
    \end{scope}
    \begin{scope}[xshift=.75in]
      \draw[thin](0,0) grid[step=\SS in](3*\SS,3*\SS);
      \foreach \x in {0,...,2}{
          \fill (\x*\SS+.125,.125)circle (.075in);
          \fill (.125,\x*\SS+.125)circle (.075in);
        }
      \draw[very thick](0,0) rectangle(3*\SS,3*\SS);
    \end{scope}
    \begin{scope}[xshift=.75in,xshift=1in]
      \draw[thin](0,0) grid[step=\SS in](4*\SS,4*\SS);
      \foreach \x in {1,...,3}{
          \fill (\x*\SS+.125,.125)circle (.075in);
          \fill (.125,\x*\SS+.125)circle (.075in);
        }
      \draw[very thick](0,0) rectangle(4*\SS,4*\SS);
    \end{scope}
    \begin{scope}[xshift=.75in,xshift=1in,xshift=1.25in]
      \draw[thin](0,0) grid[step=\SS in](5*\SS,5*\SS);
      \foreach \x in {1,...,4}{
          \fill (\x*\SS+.125,.125)circle (.075in);
          \fill (.125,\x*\SS+.125)circle (.075in);
        }
      \draw[very thick](0,0) rectangle(5*\SS,5*\SS);
    \end{scope}
  \end{tikzpicture}
  \caption{Configurations of minimum total $2$-dominating sets of $K_n\Box K_n$ for $n=2,3,4,5$.} \label{fig:Conf3-6}
\end{figure}


\begin{theorem}\label{th:nn}
  For every $n\ge2$ we have
  \begin{equation}\label{eq:nn}
    \gamma_{2t}(K_n\Box K_n) = \left\{
    \begin{array}{ll}
      (3n)/2,\quad & \text{if } n\equiv 0 \,( \text{mod }\, 4), \\
      (3n+1)/2,\quad & \text{if } n\equiv1\,( \text{mod }\, 2), \\
      (3n+2)/2,\quad & \text{if } n\equiv2\,( \text{mod }\, 4).
    \end{array}
    \right.
  \end{equation}
\end{theorem}

\begin{proof}
  First we proceed by induction on $n$ for obtaining
  \begin{equation}\label{eq:=2}
    \gamma_{2t}(K_n\Box K_n) \ge \left\{
    \begin{array}{ll}
      6k-2,\quad & \text{if  } n=4k-2, \\
      6k-1,\quad & \text{if  } n=4k-1, \\
      6k,\quad   & \text{if  } n=4k,   \\
      6k+2,\quad & \text{if  } n=4k+1.
    \end{array}
    \right.
  \end{equation}
  By Proposition \ref{prop:n=3-6}, \eqref{eq:=2} holds for $k=1$.
  Assume that \eqref{eq:=2} holds for $k=r$. Hence, Theorem \ref{th:-3o4} gives
  \[
    \begin{aligned}
      \gamma_{2t}(K_{4r+2}\Box K_{4r+2}) & \ge \min\big\{5 + \gamma_{2t}(K_{4r-1}\Box K_{4r-1}), 6+\gamma_{2t}(K_{4r-2}\Box K_{4r-2})\big\}\ge 6r+4, \\
      \gamma_{2t}(K_{4r+3}\Box K_{4r+3}) & \ge \min\big\{5 + \gamma_{2t}(K_{4r}\Box K_{4r}), 6+\gamma_{2t}(K_{4r-1}\Box K_{4r-1})\big\}\ge 6r+5,     \\
      \gamma_{2t}(K_{4r+4}\Box K_{4r+4}) & \ge \min\big\{5 + \gamma_{2t}(K_{4r+1}\Box K_{4r+1}), 6+\gamma_{2t}(K_{4r}\Box K_{4r})\big\}\ge 6r+6,     \\
      \gamma_{2t}(K_{4r+5}\Box K_{4r+5}) & \ge \min\big\{5 + \gamma_{2t}(K_{4r+2}\Box K_{4r+2}), 6+\gamma_{2t}(K_{4r+1}\Box K_{4r+1})\big\}\ge 6r+8.
    \end{aligned}
  \]
  We continue in this fashion obtaining a configuration for $S$ that yields the equality by putting in diagonal matter $k-1$ configurations of $4\times 4$ blocks and another configuration with size congruent with $n$ modulo $4$.
  In other words, build $S$ for every $n=4k+\alpha$ with $k\ge1$ and $\alpha=-2,-1,0,1$ .
  Take $S_i$ as a minimum $2$-total domination set of the induced subgraph of $\{v_{4(i-1)+1},\ldots,v_{4i}\}\times\{w_{4(i-1)+1},\ldots,w_{4i}\} $ for $i\le k-1$ (if $k>1$) and  $S_k$ as a minimum $2$-total domination set for the induced subgraph of  $\{v_{4(k-1)+1},\ldots,v_{4k+\alpha}\}\times\{w_{4(k-1)+1},\ldots,w_{4k+\alpha}\} $.
  Finally, take $S:=\displaystyle\bigcup_{i=1}^{k} S_i$ which is a total $2$-dominating set of $K_n\Box K_n$.
\end{proof}

We can use Theorem \ref{th:+4} and mathematical induction to obtain close formulas for $\{\gamma_{2t}(K_n\Box K_{n+1})\}_{n=2}^{\infty}$ 
as well as other similar results. The proof of the following proposition is also recommended to the reader.

\begin{proposition}\label{prop:nxn+1-2}
  We have $\gamma_{2t}(K_2\Box K_3)=4$, $\gamma_{2t}(K_3\Box K_4)=6$, $\gamma_{2t}(K_4\Box K_5)=7$, $\gamma_{2t}(K_5\Box K_6)=9$, $\gamma_{2t}(K_6\Box K_7)=10$. 
\end{proposition}

\begin{theorem}\label{th:nn+1}
  For every $n\ge2$ we have
  \begin{equation}\label{eq:nn+1}
    \gamma_{2t}(K_n\Box K_{n+1}) = \left\{
    \begin{array}{ll}
      (3n+2)/2,\quad & \text{if } n\equiv 0 \,( \text{mod }\, 2), \\
      (3n+3)/2,\quad & \text{if } n\equiv1\,( \text{mod }\, 2).
    \end{array}
    \right.
  \end{equation}
\end{theorem}


The following result is a direct application of the the recursive formulas in Theorems \ref{th:+4} and \ref{th:+13} with initial conditions given in Theorems \ref{th:nn} and \ref{th:nn+1}.
Note that by Proposition \ref{prop:+m_2n} we have that $\gamma_{2t}(K_n\Box K_m)=2n$ when $m\geq 3n$. Hence, for numerical reason, we can assume that $\gamma_{2t}(K_x\Box K_y)=0$ if $x\leq0$ or $y\leq0$.

\begin{theorem}\label{th:nm}
  For every $2\leq n\leq m$ we have 
  \begin{equation}\label{eq:nm}
    \gamma_{2t}(K_{n}\Box K_{m}) = \min\left\{\gamma_{2t}\left(K_{n-\left\lfloor\frac{m-n}2\right\rfloor}\Box K_{n-\left\lfloor\frac{m-n}2\right\rfloor+\left\{\left\lceil\frac{m-n}2\right\rceil-\left\lfloor\frac{m-n}2\right\rfloor\right\}}\right) + 3\left\lfloor\frac{m-n}2\right\rfloor,2n\right\}.
  \end{equation}
\end{theorem}

Hence, we can obtain a closed formula for the total $2$-domination number of $K_n\Box K_m$ for every $n,m\geq2$ using Theorem \ref{th:nm}
\[
  \gamma_{2t}(K_{n}\Box K_{m}) = \min\left\{\gamma_{2t}\left(K_{n-\left\lfloor\frac{m-n}2\right\rfloor}\Box K_{n-\left\lfloor\frac{m-n}2\right\rfloor+\left\{\left\lceil\frac{m-n}2\right\rceil-\left\lfloor\frac{m-n}2\right\rfloor\right\}}\right) + 3\left\lfloor\frac{m-n}2\right\rfloor,2n\right\}.
\]
We call $a=n-\left\lfloor\frac{m-n}2\right\rfloor$ and $b=3\left\lfloor\frac{m-n}2\right\rfloor$. Now, we have two cases, the case where $n,m$ have the same parity, \emph{i.e.}, $n\equiv m \, (\text{mod } 2)$, and the case where $n,m$ have opposing parity, \emph{i.e.}, $n\not\equiv m \, (\text{mod } 2)$. To refine that idea we let consider the quotient and remainder of $n,m$ by dividing by $8$, that is $n=8q_n+r_n$ and $m=8q_m+r_m$, for $0\leq r_n,r_m<8$. We check these in two cases.

Let $r_n\equiv r_m\, (\text{mod } 2)$. Then, $a=8q_n+r_n-\frac{8q_m+r_m-8q_n-r_n}{2}=12q_n-4q_m+\frac{3r_n-r_m}{2}$ and $b=12(q_m-q_n) + 3\cdot\frac{r_m-r_n}2$.
Now we find the quotient and remainder for the divisor of 4 to use previous theorems, so we get that $a= 4\left(3q_n-q_m+\left\lfloor\frac{\left(\tfrac{3r_n-r_m}{2}\right)}{4}\right\rfloor\right) +r_a$ where $0\leq r_a<4$ and $r_a\equiv \frac{3r_n-r_m}{2}\, (\text{mod } 4)$.
Now we use Theorem \ref{th:nn} to find $\gamma_{2t}(K_a\Box K_a)$ to get, for $a=4q_a+r_a$ where $q_a=\left(3q_n-q_m+\left\lfloor\frac{\left(\tfrac{3r_n-r_m}{2}\right)}{4}\right\rfloor\right)$
\[
  \gamma_{2t}(K_a\Box K_a)=
  \left\{
  \begin{array}{ll}
    6q_a    & r_a\equiv 0 \mod 4 \\
    6q_a +2 & r_a\equiv 1 \mod 4 \\
    6q_a +4 & r_a\equiv 2 \mod 4 \\
    6q_a+5  & r_a\equiv 3 \mod 4 \\
  \end{array}
  \right.
\]

In the case where $r_n\not\equiv r_m \, (\text{mod } 2)$, we have $a=8q_n+r_n-\lfloor\frac{8q_m+r_m-8q_n-r_n}{2}\rfloor=12q_n-4q_m+\frac{3r_n-r_m-1}{2}$ and $b=12(q_m-q_n) + 3\cdot\frac{r_m-r_n-1}2$. We follow the same method above instead using Theorem \ref{th:nn+1} with the expanded expressions of
\[
  \gamma_{2t}(K_a\Box K_{a+1})=
  \left\{
  \begin{array}{ll}
    6q_a+1 & r_a\equiv 0 \mod 4 \\
    6q_a+3 & r_a\equiv 1 \mod 4 \\
    6q_a+4 & r_a\equiv 2 \mod 4 \\
    6q_a+6 & r_a\equiv 3 \mod 4 \\
  \end{array}
  \right.
\]

Combining results above, we have for any given $n,m$, a look up Table \ref{closed_formula} for the total 2-domination number of $K_n\Box K_m$.

So for
\[
  \mathcal Q= 6q_n+6q_m 
\]

\begin{table}[h!]
  \centering
  \begin{tabular}{|c|r|r|r|r|r|r|r|r|}
    \hline
                       & $m\equiv0$       & $m\equiv1$       & $m\equiv2$       & $m\equiv3$       & $m\equiv4$       & $m\equiv5$       & $m\equiv6$        & $m\equiv7$        \\
    \hline
    $n\equiv 0 \mod 8$ & $\mathcal Q+ 0 $ & $\mathcal Q+ 1 $ & $\mathcal Q+ 2 $ & $\mathcal Q+ 3 $ & $\mathcal Q+ 4 $ & $\mathcal Q+ 4 $ & $\mathcal Q+ 5 $  & $\mathcal Q+ 6 $  \\
    $n\equiv 1 \mod 8$ & $\mathcal Q+ 1 $ & $\mathcal Q+ 2 $ & $\mathcal Q+ 3 $ & $\mathcal Q+ 3 $ & $\mathcal Q+ 4 $ & $\mathcal Q+ 5 $ & $\mathcal Q+ 6 $  & $\mathcal Q+ 7 $  \\
    $n\equiv 2 \mod 8$ & $\mathcal Q+ 2 $ & $\mathcal Q+ 3 $ & $\mathcal Q+ 4 $ & $\mathcal Q+ 4 $ & $\mathcal Q+ 5 $ & $\mathcal Q+ 6 $ & $\mathcal Q+ 6 $  & $\mathcal Q+ 7 $  \\
    $n\equiv 3 \mod 8$ & $\mathcal Q+ 3 $ & $\mathcal Q+ 3 $ & $\mathcal Q+ 4 $ & $\mathcal Q+ 5 $ & $\mathcal Q+ 6 $ & $\mathcal Q+ 7 $ & $\mathcal Q+ 7 $  & $\mathcal Q+ 8 $  \\
    $n\equiv 4 \mod 8$ & $\mathcal Q+ 4 $ & $\mathcal Q+ 4 $ & $\mathcal Q+ 5 $ & $\mathcal Q+ 6 $ & $\mathcal Q+ 6 $ & $\mathcal Q+ 7 $ & $\mathcal Q+ 8 $  & $\mathcal Q+ 9 $  \\
    $n\equiv 5 \mod 8$ & $\mathcal Q+ 4 $ & $\mathcal Q+ 5 $ & $\mathcal Q+ 6 $ & $\mathcal Q+ 7 $ & $\mathcal Q+ 7 $ & $\mathcal Q+ 8 $ & $\mathcal Q+ 9 $  & $\mathcal Q+ 9 $  \\
    $n\equiv 6 \mod 8$ & $\mathcal Q+ 5 $ & $\mathcal Q+ 6 $ & $\mathcal Q+ 6 $ & $\mathcal Q+ 7 $ & $\mathcal Q+ 8 $ & $\mathcal Q+ 9 $ & $\mathcal Q+ 10 $ & $\mathcal Q+ 10 $ \\
    $n\equiv 7 \mod 8$ & $\mathcal Q+ 6 $ & $\mathcal Q+ 7 $ & $\mathcal Q+ 7 $ & $\mathcal Q+ 8 $ & $\mathcal Q+ 9 $ & $\mathcal Q+ 9 $ & $\mathcal Q+ 10 $ & $\mathcal Q+ 11 $ \\
    \hline
  \end{tabular}
  \caption{Table with the minimum values of the total $2$-domination number of $K_n\Box K_m$ for $n,m\geq2$.}\label{closed_formula}
\end{table}

Thus, we can take the minimum of table entry above or $2n$ to get our value for $\gamma_{2t}(K_{n}\Box K_{m})$.


Theorems \ref{th:nn}, \ref{th:nn+1} and \ref{th:nm} have a direct consequence which is a general result which improves the result in \cite[Proposition 3.3]{BSS} for $k=2$.

\begin{theorem}\label{th:lowerbound}
  Let $G,H$ be two graphs without isolate vertex and order $n$ and $m$ respectively. Then
  \[
    \gamma_{2t}(G\Box H) \ge \displaystyle\frac{3}{2}\, \min\{n,m\}.
  \]
  Furthermore, if $n\le m$
  \[
    \displaystyle\liminf_{n\to \infty} \frac{\gamma_{2t}(G\Box H)}{n} = \frac32.
  \]
\end{theorem}


%

From the closed formula for $\{\gamma_{2t}(K_n\Box K_m)\}_{n,m\ge2}$ described in Table \ref{closed_formula} we have the following asymptotic behavior of the sequence.

\begin{theorem}\label{th:behaviornm}
  For natural numbers $n,m\ge2$ and a real number $\frac35\le\lambda\le\frac53$, if both $n$ and $m$ increase without bound with certain behavior such that $\displaystyle\frac{m}{n}\to\lambda$, then
  \begin{equation}\label{eq:behaviornm}
    \displaystyle \frac{\gamma_{2t}(K_n\Box K_m)}{n} \,\xrightarrow{\quad n,m\to\infty,\frac{m}{n}\to\lambda\quad} \, \frac{3(1+\lambda)}{4}.
  \end{equation}
\end{theorem}

Note that restriction for $\lambda$ follows from the fact $\gamma_{2t}(K_n\Box K_m)\le 2\min\{n,m\}$.

\begin{proof}
  The result follows directly from the facts that $n=8 q_n+r_n$ and $m=8 q_m+r_m$ with $0\le r_n,r_m<8$, and that from Table \ref{closed_formula} we have $\mathcal{Q}=\mathcal{Q}_{n,m}=6(q_n+q_m)$. Thus, we have
  \[
    \displaystyle \frac{\gamma_{2t}(K_n\Box K_m)}{n}=\frac{6(q_n+q_m)+s_{n,m}}{8 q_n+r_n}\quad \text{for some } 0\le s_{n,m}\le11.
  \]
  Therefore, we have
  \[
    \displaystyle \frac{\gamma_{2t}(K_n\Box K_m)}{n} \to \frac{6(q_n+q_m)+s_{n,m}}{8 q_n+r_n} \to \frac{6(1+\frac{q_m}{q_n})+\frac{s_{n,m}}{q_n}}{8 +\frac{r_n}{q_n}}  \to  \frac{3(1+\lambda)}{4}.
  \]
\end{proof}

\
\section{Total $3$-domination number of $K_n\Box K_m$}\label{sect:k=3}


In this section we investigate the case $k=3$. Given the geometric representation of $\gamma_{kt}\left(K_n\Box K_m\right)$, we conjecture the parity of $k$ will be important.
The following result improves Lemma \ref{l:UpperB} when $k=3$.

\begin{proposition}\label{prop:mink3}
  For every $m\ge n\ge3$ we have $\gamma_{3t}(K_n\Box K_m)\geq 2n+2$.
\end{proposition}

\begin{proof}
  Consider a minimum rook configuration in an $n\times m$ board.
  Notice that by Lemmas \ref{l:UpperB} and \ref{l:rooks} we have that there is at least a rook in each row and column.

  \smallskip

  Assume first that there is a row (column, resp.) with exactly one rook.
  Note that the squares in that row (column, resp.) without rook is dominated by at least other two rooks in the same column (row, resp.) and the square with the rook is dominated by at least other three rooks in the same column (row, resp.). Thus, we can count the rooks by columns (rows, resp.) and obtain that the number of rooks is at least $2m+2$ ($2n+2$, resp.).

  \smallskip

  Assume now that every row and column has at least two rooks. Seeking for a contradiction assume that the number of rooks is less than $2n+2$. Hence, we have that $n=m$. Now consider a square with one of the rooks. Since the square is dominated by other three rooks we have a row or column with exactly three rooks, and consequently, the total number of rooks is $2n+1$ and there is only one row (column, resp.) with exactly three rooks. Without loss of generality we can assume that the first row and column have three rooks each. Since $6<2n+1$, there is at least one rook out of firsts row and column; however, the square with that rook is dominating by other three rooks that is a contradiction with the only one row and column with three rooks are the firsts. Therefore, we have the result.
\end{proof}

The following result is a version of Theorem \ref{th:rows-cols} for $k=3$ with a lower upper bound.

\begin{proposition}\label{prop:nmk3}
  For every $2\le k< n\le m$, we have
  \[
    \gamma_{3t}(K_n\Box K_m)\leq (2m+n)-2\,\left\lfloor\frac{m-1}{3}\right\rfloor.
  \]
  Furthermore, if $n=m$, we have
  \[
    \gamma_{3t}(K_n\Box K_n)\leq 3n-2\,\left\lfloor\frac{n-1}{3}\right\rfloor.
  \]
\end{proposition}

\begin{proof}
  Start with the grid representation of $K_n\Box K_m$ with rooks populating like in Figure \ref{fig:nmk3}, \emph{i.e.}, put two rooks in the firsts $m-1$ columns as blocks of size $2\times 3$ (perhaps with one block of $2\times4$ or $2\times5$ when $m-1$ is not a multiple of $3$) like steps in a stair, and a block in the last column like a wall from the top of the stair up to the top right corner of the board (perhaps, up to two more rooks at the level of the upper step as needed in the case $n=m=4$). Note that gray rooks in Figure \ref{fig:nmk3} will be needed in some cases depending of particular $n$ and $m$.
  Therefore, we can easily verify that every square is dominated by at least three rooks.
  \[
    \gamma_{3t}(K_n\Box K_m)\leq 2\,m +n-2\,\left\lfloor\frac{m-1}{3}\right\rfloor.
  \]

  \begin{figure}[H]
    \centering
    \begin{tikzpicture}[x=.8in, y=.8in]
      \draw[thin,step=.25] (0,0)grid(4,4);
      \draw[thick] (0,0)rectangle(4,4);
      \foreach \x/\y in {
          $c_1$/0,
          $\cdots$/.25,
          $\cdots$/3.5,
          $c_m$/3.75
        }{
          \node[anchor=south] at (\y +.125,4) {\x};
        }
      \foreach \x/\y in {
          $f_1$/0,
          $\vdots$/.25,
          $\vdots$/3.5,
          $f_n$/3.75
        }{
          \node[anchor=east] at (0,3.875-\y) {\x};
        }
      \foreach \x in {0,1,2}{
          \fill (.125+.25*\x,.125) circle(.075);
          \fill (.125+.25*\x,.375) circle(.075);
          \fill (.875+.25*\x,.625) circle(.075);
          \fill (.875+.25*\x,.875) circle(.075);
          \fill (1.625+.25*\x,1.125) circle(.075);
          \fill (1.625+.25*\x,1.375) circle(.075);
          \foreach \x in {1,2,3,5,6,7}{\fill (2.25+\x/16,1.5+\x/16) circle (1pt);}
          \fill (3.125+.25*\x,2.125) circle(.075);
          \fill (3.125+.25*\x,2.375) circle(.075);
          \fill[color=gray] (2.875,2.125) circle(.075);
          \fill[color=gray] (2.875,2.375) circle(.075);
          \fill[color=gray] (3.875,2.125) circle(.075);
          \fill[color=gray] (3.875,2.375) circle(.075);
          \fill (3.875,2.625+.25*\x) circle(.075);
          \fill (3.875,3.375+.25*\x) circle(.075);
        }
    \end{tikzpicture}
    \caption{Total $3$-dominating configuration for $K_n\Box K_m$.} \label{fig:nmk3}
  \end{figure}
\end{proof}

The result above has the following consequence that improves Theorem \ref{th:rows-colsnxn} for $k=3$. It will be improved by Theorem \ref{th:behaviork3}.

\begin{corollary}\label{cor:behaviork3}
  Let $G,H$ be two graphs without isolate vertex and order $n$ and $m$, respectively. If $3< n\le m$ then
  \[
    \displaystyle\liminf_{n\to \infty}\; \frac{\gamma_{kt}(G\Box H)}{n} \le \frac{7}3.
  \]
\end{corollary}

\begin{proof}
  By obvious reason the minimum values of $\gamma_{kt}(G\Box H)$ is attained by $\gamma_{kt}(K_n\Box K_m)$. Besides, by natural graph inclusion without loss of generality we can assume that $n=m$. So, Proposition \ref{prop:nmk3} gives
  \[
    \gamma_{3t}(K_n\Box K_n)\leq 3n-2\,\left\lfloor\frac{n-1}{3}\right\rfloor,
  \]
  Thus, the result follows from taking limit as $n$ approaches to infinity on  $\displaystyle\frac{\gamma_{kt}(K_n\Box K_n)}{n}$.
\end{proof}

The configuration in Figure \ref{fig:nmk3} gives minimum total $3$-dominating sets for $n=4,5,\ldots,13$. In particular we have $\gamma_{3t}(K_4\Box K_4)=10$, $\gamma_{3t}(K_5\Box K_5)=12$, $\gamma_{3t}(K_6\Box K_6)=14$, $\gamma_{3t}(K_7\Box K_7)=16$, $\gamma_{3t}(K_8\Box K_8)=18$, $\gamma_{3t}(K_9\Box K_9)=21$, $\gamma_{3t}(K_{10}\Box K_{10})=22$, $\gamma_{3t}(K_{11}\Box K_{11})=25$, $\gamma_{3t}(K_{12}\Box K_{12})=28$ and $\gamma_{3t}(K_{13}\Box K_{13})=29$; \emph{i.e.},

\begin{equation}\label{eq:nnk3}
  \gamma_{3t}(K_n\Box K_n)=\left\{\begin{aligned} 2n+2 \qquad & \text{ for } n=4,5,6,7,8,10 \\
    2n+3 \qquad & \text{ for } n= 9,11,13     \\
    2n+4 \qquad & \text{ for } n= 12.
  \end{aligned}\right.
\end{equation}

\begin{theorem}\label{prop:nn3}
  For every $n\ge4$, we have
  \begin{equation}\label{eq:nn3}
    \gamma_{3t}(K_n\Box K_n)\leq\left\{
    \begin{aligned}
      22r+10 \qquad & \text{ for } n=10r+4   \\
      22r+12 \qquad & \text{ for } n=10r+5   \\
      22r+14 \qquad & \text{ for } n=10r+6   \\
      22r+16 \qquad & \text{ for } n=10r+7   \\
      22r+18 \qquad & \text{ for } n=10r+8   \\
      22r+21 \qquad & \text{ for } n=10r+9   \\
      22r+22 \qquad & \text{ for } n=10r+10  \\
      22r+25 \qquad & \text{ for } n=10r+11  \\
      22r+28 \qquad & \text{ for } n=10r+12  \\
      22r+29 \qquad & \text{ for } n=10r+13.
    \end{aligned}\right.
  \end{equation}
\end{theorem}

\begin{proof}
  Start with the grid representation of $K_n\Box K_n$.
  Configurations like the one in Figure \ref{fig:nmk3} for $k=3$ and $n=m$ holding values in $\{4,5,\ldots,13\}$ gives minimum number of rooks in a total $3$-dominating set of $K_n\Box K_n$ when $r=0$. Let us consider $r\ge1$.
  Now consider a finite sequence of $r$ identical configurations with size $10\times10$, like in Figure \ref{fig:nmk3}, locate into the main diagonal and add the minimum configuration for $n-10r$ into the main diagonal, too.
  This configuration is total $3$-dominating, and consequently, we obtain the desire result.
\end{proof}

We believe that \eqref{eq:nn3} gives a closed formula for $\gamma_{3t}(K_n\Box K_n)$ for $n\ge4$, but, it remains as an open problem on the topic.

\begin{corollary}\label{cor:nn3}
  For every $n\ge14$, we have
  \begin{equation}\label{eq_cor:nn3}
    \displaystyle\frac{\gamma_{3t}(K_n\Box K_n)}{n}< 2+\frac13.
  \end{equation}
\end{corollary}

\begin{lemma}\label{l:4x1}
  For every $n\ge 4$, there is a minimum configuration of Problem \ref{prob:board} for $k=3$ and $n=m$ where one of the rows (columns, resp.) has only one rook, and consequently, one of the columns (rows, resp.) has at least $4$ rooks.
\end{lemma}

\begin{proof}
  Start with the grid representation of $K_n\Box K_n$.
  Configurations of rooks placed into an $n\times n$ board like the one in Figure \ref{fig:nmk3} for $n$ holding values in $\{4,5,\ldots,13\}$ gives minimum rooks for Problem \ref{prob:board} with at least a row with exactly one rook. So, we can assume $n\ge14$.
  Now seeking for a contradiction we can assume that there is a minimum configuration that is a solution of Problem \ref{prob:board} where there are neither a row or column with exactly one rook.
  By \eqref{eq_cor:nn3}, we have that
  $\gamma_{3t}(K_n\Box K_n)\leq 7n/3$.
  Thus, by Lemma \ref{l:rooks}, there is at least a rook in each row and each column. Hence, we have there are at least $2$ rooks in each row and each column.
  Let $r_1$ ($r_2$, resp.) be the number of rows (columns, resp.) with exactly $2$ rooks. Notice that \eqref{eq_cor:nn3} gives $r_1,r_2\ge2n/3$ since no rows and no columns has exactly one rook. However, the squares in the intersection of the $r_1$ rows and the $r_2$ columns with exactly $2$ rooks have no rooks located therein, and consequently, the total number of rooks in the board is greater than or equal $2r_1+2r_2\ge 8n/3$.
  That is the contradiction we were looking for, and consequently, we obtain the result for $n\ge14$, too.
\end{proof}

\begin{theorem}\label{th:+14}
  For every $4\leq n\leq m$, we have
  \begin{equation}\label{eq:+14}
    \gamma_{3t}(K_{n+1}\Box K_{n+4}) \le \gamma_{3t}(K_{n}\Box K_{n})+4.
  \end{equation}
\end{theorem}

\begin{proof}
  Let $S^\prime$ be a minimum total $3$-dominating set of $K_n\Box K_m$. By Lemma \ref{l:4x1} there is a row with at least $4$ (four) rooks.
  Let us consider $A=\{(v_{n+1},w_{n+1})$, $(v_{n+1},w_{n+2})$, $(v_{n+1},w_{n+3})$, $(v_{n+1},w_{n+4})\}$.
  Thus it is a simple matter to check that $S^\prime\cup A$ is a total $3$-dominating set of $K_{n+1}\Box K_{m+4}$ obtaining $\gamma_{3t}(K_{n+1}\Box K_{m+4})\leq \gamma_{3t}(K_n\Box K_m)+4$.
\end{proof}

The following result improves Theorem \ref{th:rows-colsnxn} when $k=3$.

\begin{theorem}\label{th:behaviork3}
  Let $G,H$ be two graphs without isolate vertex and order $n$ and $m$, respectively. If $3< n\le m$ then
  \[
    \displaystyle\liminf_{n\to \infty}\; \frac{\gamma_{3t}(G\Box H)}{n} \le \frac{11}5.
  \]
\end{theorem}

\begin{proof}
  By obvious reason the minimum values of $\gamma_{kt}(G\Box H)$ is attained by $\gamma_{kt}(K_n\Box K_m)$ and $m=n$. So, Proposition \ref{prop:nn3} gives
  \[
    \displaystyle \frac{\gamma_{3t}(K_n\Box K_n)}{n}\le\frac{22r+x}{10r+y}
  \]
  with $r$ as in Theorem \ref{prop:nn3} and for some $10\le x\le 29$ and some $4\le y\le 13$.
  Thus, the result follows by taking (inferior) limit as $n$ approaches to infinity.
\end{proof}

Natural open problems in this context are

\begin{openprob}
  Find a closed formula for $\gamma_{kt}(K_n\Box K_m)$ for every even number $k\ge2$.
\end{openprob}

\begin{openprob}
  Find a closed formula for $\gamma_{kt}(K_n\Box K_n)$, or
  \[
    \displaystyle\lim_{n\to\infty}\frac{\gamma_{kt}(K_n\Box K_n)}{n} \quad \text{ for every odd number } k\ge3.
  \]
\end{openprob}

\section*{Acknowledgements}

The first author was supported by a grant from Agencia Estatal de Investigaci\'on (PID2019-106433GB-I00 /AEI/10.13039/501100011033), Spain.

\

\end{document}